\documentclass[11pt]{amsart}
\usepackage{amsmath}
\usepackage[applemac]{inputenc}
\usepackage[left=3cm,marginparwidth=2cm]{geometry}
\usepackage{enumerate}
\usepackage{amsmath}
\usepackage{amssymb}
\usepackage{epic}
\usepackage{eepic}
\usepackage{amssymb}
\usepackage{color}

\usepackage{microtype}

\makeindex

\newcommand{\A}{\mathcal{A}}

\newcommand{\NCT}{\mathcal T}

\newcommand{\Cstar}{\ensuremath{{\mathrm C}^\ast}}

\newcommand{\ue}{\textrm{e}}
\newcommand{\ui}{\textrm{i}}

\newtheorem*{proposition*}{Proposition}

\theoremstyle{definition}
\newtheorem*{remark*}{Remark}

\begin{document}

\title[Nonsimplicity of certain universal $\Cstar$-algebras ]{Nonsimplicity of certain universal $\boldsymbol\Cstar$-algebras}


\author{Marcel de Jeu}
\address{Mathematical Institute\\
Leiden University\\
P.O.\ Box 9512\\
2300 RA Leiden\\
the Netherlands}
\email{mdejeu@math.leidenuniv.nl}

\author{Rachid El Harti}
\address{Department of Mathematics and Computer Sciences\\
Faculty of Sciences and Techniques\\
University Hassan I \\
BP 577 Settat\\
Morocco}
\email{rachid.elharti@uhp.ac.ma}

\author{Paulo R.\ Pinto}
\address{Department of Mathematics\\
CAMGSD\\
Instituto Superior T\'{e}cnico\\
 University of Lisbon\\
Av.\ Rovisco Pais 1\\
1049-001 Lisboa\\
 Portugal}
\email{ppinto@math.tecnico.ulisboa.pt}

\subjclass[2010]{Primary 46L99; Secondary 22D25}
\keywords{Universal \Cstar-algebra, nonsimplicity}




\begin{abstract}
Given $n\geq 2$, $z_{ij}\in\mathbb{T}$ such that $z_{ij}=\overline z_{ji}$ for $1\leq i,j\leq n$ and $z_{ii}=1$ for $1\leq i\leq n$, and integers $p_1,...,p_n\geq 1$,
we show that the universal \Cstar-algebra generated by unitaries $u_1,...,u_n$ such that $u_i^{p_i}u_j^{p_j}=z_{ij}u_j^{p_j}u_i^{p_i}$ for $1\leq i,j \leq n$ is not simple if at least one exponent $p_i$ is at least two. We indicate how the method of proof by `working with various quotients' can be used to establish nonsimplicity of universal \Cstar-algebras in other cases.
\end{abstract}

\maketitle

Let $n\geq 1$, let $\theta=(\theta_{ij})$ be a skew symmetric real $n\times n$ matrix, and let $z$ be the matrix defined by $z_{ij}=\ue^{2\pi \ui \theta_{ij}}$ for $1\leq i,j\leq n$. The $n$-dimensional noncommutative torus $\NCT_z$ is the universal \Cstar-algebra that is generated by unitaries $u_1,\ldots,u_n$ such that $u_i u_j =z_{ij}u_j u_i$ for $1\leq i,j\leq n$. It is known that $\NCT_z$ is simple if and only if the matrix $\theta$ is nondegenerate, i.e.\ if and only if it has the property that, whenever $x\in\mathbb{Z}^n$ satisfies $\ue^{2\pi \ui \langle x, \theta y\rangle}=1$ for all $y\in \mathbb{Z}^n$, then $x=0$; see \cite[Theorem~1.9]{phillips1} and \cite[Theorem~3.7]{Slawny}.

The \Cstar-algebra $\NCT_z$ is a deformation of the group \Cstar-algebra of $\mathbb Z^n$. It seems natural to consider other families of such deformed group \Cstar-algebras, and, in particular, universal \Cstar-algebras that are obtained by allowing higher powers in the relations for $\NCT_z$. Therefore, given $n\geq 2$ (the case $n=1$ is clear), $z_{ij}\in\mathbb{T}$ such that $z_{ij}=\overline z_{ji}$ for $1\leq i,j\leq n$ and $z_{ii}=1$ for $1\leq i\leq n$, and integers $p_1,...,p_n\geq 1$,
we let $\A_{z,p_1,...,p_n}$ be the universal \Cstar-algebra that is generated by unitaries $u_1,...,u_n$ such that
\begin{equation*}
u_i^{p_i}u_j^{p_j}=z_{ij}u_j^{p_j}u_i^{p_i}\quad \hbox{for}\ \ 1\leq i,j \leq n.
\end{equation*}
Assuming that at least one of the $p_i$ is at least two, when is $\A_{z,p_1,...,p_n}$ simple?

The most natural first approach to this question seems to be one along the lines in \cite{phillips1, Slawny}. When attempting this, it soon becomes clear that the higher exponents cause serious complications. It may therefore come as a pleasant surprise\textemdash at least it did so to the present authors\textemdash that, given the fact that the noncommutative tori are nonzero, a purely algebraic argumentation can be employed to show that $\A_{z,p_1,...,p_n}$ is never simple. The argument is so elementary that it could even easily be overlooked. After all, it does not appear to be immediate how the fact that $-1$ has two different complex square roots can be put to good use to show that the universal \Cstar-algebra that is generated by unitaries $u_1,u_2$ such that $u_1^2 u_2=-u_2 u_1^2$ is not simple; yet this is still the case. Since a similar argument will work in various other suitable contexts, it seems worthwhile to make it explicit in this short note.

\begin{proposition*}
Let $n\geq 2$ and suppose that $p_i\geq 2$ for some $i$ such that $1\leq i\leq n$. Then $\A_{z,p_1,...,p_n}$ is not simple.
\label{prop1}
\end{proposition*}

\begin{proof} We prove the proposition by contradiction, so assume that $\A_{z,p_1,...,p_n}$ is simple.
For $1\leq i,j\leq n$, choose $\rho_{ij}\in\mathbb T$ such that
\begin{equation}\label{eq:system}
\begin{cases}
\rho_{ij}^{p_ip_j}=z_{ij} &\hbox{for}\ 1\leq i,j\leq n;\\
\overline \rho_{ij}=\rho_{ji} &\hbox{for}\ 1\leq i,j\leq n;\\
\rho_{ii}=1 &\hbox{for}\ 1\leq i\leq n.
\end{cases}
\end{equation}

Let $\NCT_\rho$ be the $n$-dimensional noncommutative torus that is generated by unitaries $v_1,...,v_n$ such that $v_iv_j=\rho_{ij}v_jv_i$ for $1\leq i,j \leq n$. Since $v_i^{p_i}v_j^{p_j}=\rho_{ij}^{p_i p_j}v_j^{p_j} v_i^{p_i}=z_{ij}v_j^{p_j} v_i^{p_i}$ for $1\leq i,j\leq n$, there exists a surjective $^\ast$-homomorphism $\pi:\A_{z,p_1,...,p_n}\to\NCT_\rho$ such that $\pi(u_i)=v_i$ for $1\leq i\leq n$. Since $\NCT_\rho\not= \{\,0\,\}$, we see that $\A_{z,p_1,...,p_n}\neq\{\,0\,\}$, and also that $\ker(\pi)\neq\A_{z,p_1,...,p_n}$. Since we have assumed that $\A_{z,p_1,...,p_n}$ is simple, we conclude from the latter inequality that $\ker(\pi)=\{\,0\,\}$, so that $\pi$ is an isomorphism between $\A_{z,p_1,...,p_n}$ and $\NCT_\rho$. As a consequence, we see that $u_iu_j=\rho_{ij}u_ju_i$  for $1\leq i,j\leq n$.

Under our assumptions, there are $i$ and $j$ such that $1\leq i\neq j \leq n$ and such that the corresponding exponent $p_ip_j$ in the first line of  \eqref{eq:system} is at least 2. Therefore, there exists a solution matrix $(\rho_{ij}^\prime)$ of \eqref{eq:system} that is different from our chosen solution matrix $(\rho_{ij})$.  By the same argument as above, we also have $u_iu_j=\rho_{ij}^\prime u_ju_i$, so that $\rho_{ij}v_jv_i=\rho_{ij}^\prime v_jv_i$ for $1\leq i,j\leq n$. This implies that $(\rho_{ij}-\rho_{ij}^\prime)1_{\A_{z,p_1,...,p_n}}=0_{\A_{z,p_1,...,p_n}}$ for $1\leq i,j\leq n$. On choosing $i$ and $j$ such that $\rho_{ij}\neq\rho_{ij}^\prime$, we find that $1_{\A_{z,p_1,...,p_n}}=0_{\A_{z,p_1,...,p_n}}$. We conclude that $\A_{z,p_1,....,p_n}=\{\,0\,\}$. This contradiction shows that $\A_{z,p_1,...,p_n}$ is not simple.
\end{proof}

\begin{remark*}\quad
\begin{enumerate}
\item In spite of the elementary nature of the above proof, the result in itself is still not trivial, as it is based on the fact, used in an essential way in the proof, that the noncommutative tori are nonzero.

\item
One can vary the definition of the algebra $\A_{z,p_1,...,p_n}$ in the proposition by:
\begin{enumerate}
\item requiring that some of the generators are isometries, or a partial isometries, and/or
\item removing some (or even all) of the relations $u_i^{p_i}u_j^{p_j}=z_{ij}u_j^{p_j}u_i^{p_i}$.
\end{enumerate}
 Since the resulting universal \Cstar-algebra has $\A_{z,p_1,...,p_n}$ as a quotient that is not simple, it is not simple itself.

For example, for $z\in\mathbb T$, let $\mathcal{B}_z$ be the universal \Cstar-algebra that is generated by a partial isometry $v_1$, an isometry $v_2$, and a unitary $v_3$ such that $v_3v_2=z v_2v_3$. Then $\mathcal{B}_z$ is not simple. Indeed, the universal \Cstar-algebra that is generated by unitaries $u_1,u_2,u_3$ such that
\begin{align*}
u_3u_1^2&=u_1^2u_3\\
 u_3u_2 &=z u_2u_3\\
 u_2u_1^2 &=u_1^2 u_2
 \end{align*}
    is a nonsimple quotient of $\mathcal B_z$. The higher exponents, responsible for the nonsimplicity of $\mathcal B_z$, are not present in the initial relations, but they do occur in those for the quotient.
\end{enumerate}
\end{remark*}

In general, let us assume that we have a collection $\{\,\mathcal R_i : i\in I\,\}$ of sets $\mathcal R_i$ of relations for a common set of symbols $\mathcal G$ for elements of a \Cstar-algebra, such that each set of relations $\mathcal R_i$ implies one fixed set of relations $\mathcal R$. Let us also assume that the universal \Cstar-algebra $\Cstar(\mathcal R_i)$ for each set of relations $\mathcal R_i$ exists, and is nonzero. Then the universal \Cstar-algebra $\Cstar(\mathcal R)$ also exists, has each $\Cstar(\mathcal R_i)$ as a quotient, and is nonzero. If $\Cstar(\mathcal R)$ is simple, then these quotient maps are isomorphisms. Since they send generators to generators, the relations from \emph{all} sets $\mathcal R_i$ will then hold for the generators of $\Cstar(\mathcal R)$. If one can show that the \emph{simultaneous} validity of these sets of relations (each of which results from a different quotient) leads to a contradiction, this will prove that $\Cstar(\mathcal R)$ is not simple.

The above proof of the proposition employs this technique of working with various quotients. As a further example, still using unitaries, consider the universal \Cstar-algebra $\A$ that is generated by unitaries $u$ and $v$ satisfying $u^4 v=-v^3 u^7 v^2 u^7$. We shall show that $\A$ is not simple. To this end, consider the universal \Cstar-algebras $\A_\pm$ that are generated by unitaries $u$ and $v$ such that $u^2v=\pm \ui v^3 u^7$. Then $\A_\pm\neq\{\,0\,\}$. Indeed, let $W$ be any nonzero unitary operator on a Hilbert space, and put $U_\pm=\ue^{\mp\pi\ui/10}W^2$ and $V_\pm=W^{-5}$. Then $U_\pm$ and $V_\pm$ are nonzero unitary operators satisfying the relations for $\A_\pm$. Consequently, $\A_\pm\neq\{\,0\,\}$. Now note that the relations for $\A_+$ and $\A_-$ both imply the relation for $\A$, so that $\A$ has $\A_+$ and $\A_-$ as canonical quotients. In particular, $\A\neq\{\,0\,\}$. Assuming that  $\A$ is simple, one finds that $u^2v=\ui v^3 u^7$ as well as $u^2v=-\ui v^3 u^7$ for $u,v\in\A$. This leads to $2\ui 1_\A=0_\A$, so that $1_\A=0_\A$ and $\A=\{\,0\,\}$. The latter contradiction shows that $\A$ cannot be simple.

\vskip .3cm

\subsection*{Acknowledgements}
We thank the  anonymous referee for the precise reading of the manuscript, and for providing the argument showing that the algebras $\A_{\pm}$ in the final paragraph are nonzero.
The last author was partially funded by FCT/Portugal through projects UID/MAT/04459/2013 and EXCL/MAT-GEO/0222/2012.



\end{document}